\documentclass[12pt,reqno]{amsart}

\usepackage[T2A]{fontenc}
\usepackage[russian,english]{babel}

\usepackage[margin=1in]{geometry}
\usepackage{amsmath,amssymb,mathtools,esint}
\usepackage[expansion=false]{microtype}

\usepackage{tikz}
\usetikzlibrary{arrows.meta}
\usepackage{float}
\usepackage{enumitem}

\usepackage[
 colorlinks=true,
 linkcolor=blue,
 citecolor=blue,
 urlcolor=blue
]{hyperref}

\definecolor{rfblue}{RGB}{228,236,247}
\definecolor{rfgold}{RGB}{248,239,211}
\definecolor{rfgreen}{RGB}{226,242,232}
\definecolor{rfviolet}{RGB}{126,73,132}
\definecolor{rfteal}{RGB}{220,70,20}
\tikzset{
 outline/.style={draw=black,line width=.4pt,line cap=round,line join=round},
 body/.style={outline,shade,top color=rfblue!24,bottom color=rfblue!78},
 collar region/.style={shade,top color=rfgold!30,bottom color=rfgold!92},
 cap region/.style={outline,shade,top color=rfgreen!28,bottom color=rfgreen!90},
 loop/.style={draw=rfviolet,line width=.62pt,line cap=round,line join=round},
 shortcut/.style={draw=rfteal,line width=.62pt,line cap=round,line join=round},
 loop label/.style={text=rfviolet},
 shortcut label/.style={text=rfteal},
 hidden/.style={outline,dash pattern=on 4pt off 4pt},
 leader/.style={outline,dash pattern=on 1pt off 3pt},
 dashdot/.style={outline,dash pattern=on 4pt off 2pt on 1pt off 2pt},
 process/.style={outline,-{Latex[length=4pt,width=3pt]}},
 point/.style={circle,draw=black,fill=black,inner sep=0pt,minimum size=3pt},
 clear label/.style={fill=white,fill opacity=.92,text opacity=1,
 rounded corners=.5pt,inner sep=1.1pt},
 every node/.style={inner sep=1pt}
}
\newcommand{\systolepostbody}{%
 \draw[body]
 (.18,0)
 .. controls (.24,.72) and (.88,1.05) .. (1.65,1.04)
 .. controls (2.52,1.02) and (2.92,.66) .. (3.38,.47)
 .. controls (3.72,.31) and (4.05,.35) .. (4.35,.35)
 -- (5.20,.35)
 .. controls (6.02,.35) and (6.36,.16) .. (6.55,0)
 .. controls (6.36,-.16) and (6.02,-.35) .. (5.20,-.35)
 -- (4.35,-.35)
 .. controls (4.05,-.35) and (3.72,-.31) .. (3.38,-.47)
 .. controls (2.92,-.66) and (2.52,-1.02) .. (1.65,-1.04)
 .. controls (.88,-1.05) and (.24,-.72) .. (.18,0);
 \path[fill=white]
 (1.02,0) .. controls (1.25,.17) and (1.92,.17) .. (2.17,0)
 .. controls (1.92,-.17) and (1.25,-.17) .. cycle;
 \draw[outline]
 (1.02,0) .. controls (1.25,-.17) and (1.92,-.17) .. (2.17,0);
 \draw[hidden]
 (1.02,0) .. controls (1.25,.17) and (1.92,.17) .. (2.17,0);
 \path[collar region]
 (4.35,.35) -- (5.20,.35) -- (5.20,-.35) -- (4.35,-.35) -- cycle;
 \draw[cap region]
 (5.20,.35) .. controls (6.02,.35) and (6.36,.16) .. (6.55,0)
 .. controls (6.36,-.16) and (6.02,-.35) .. (5.20,-.35) -- cycle;
 \draw[outline] (4.35,.35) arc[start angle=90,end angle=270,
 x radius=.12,y radius=.35];
 \draw[hidden] (4.35,-.35) arc[start angle=-90,end angle=90,
 x radius=.12,y radius=.35];
 \draw[outline] (5.20,.35) arc[start angle=90,end angle=270,
 x radius=.11,y radius=.35];
 \draw[hidden] (5.20,-.35) arc[start angle=-90,end angle=90,
 x radius=.11,y radius=.35];%
}

\newcommand{\systoleprebody}{%
\draw[body]
(.18,0)
.. controls (.24,.72) and (.88,1.05) .. (1.65,1.04)
.. controls (2.52,1.02) and (2.92,.66) .. (3.38,.47)
.. controls (3.72,.31) and (4.05,.35) .. (4.35,.35)
-- (5.20,.35)
.. controls (5.64,.32) and (6.10,.16) .. (6.50,.07)
-- (6.50,-.07)
.. controls (6.10,-.16) and (5.64,-.32) .. (5.20,-.35)
-- (4.35,-.35)
.. controls (4.05,-.35) and (3.72,-.31) .. (3.38,-.47)
.. controls (2.92,-.66) and (2.52,-1.02) .. (1.65,-1.04)
.. controls (.88,-1.05) and (.24,-.72) .. (.18,0);

\path[collar region]
(4.35,.35) -- (5.20,.35)
-- (5.20,-.35) -- (4.35,-.35) -- cycle;

\path[shade,top color=black!2,bottom color=black!12]
(5.20,.35)
.. controls (5.64,.32) and (6.10,.16) .. (6.50,.07)
-- (6.50,-.07)
.. controls (6.10,-.16) and (5.64,-.32) .. (5.20,-.35)
-- cycle;

\draw[outline]
(4.35,.35) -- (5.20,.35)
.. controls (5.64,.32) and (6.10,.16) .. (6.50,.07)
-- (6.50,-.07)
.. controls (6.10,-.16) and (5.64,-.32) .. (5.20,-.35)
-- (4.35,-.35);

\path[fill=white]
(1.02,0) .. controls (1.25,.17) and (1.92,.17) .. (2.17,0)
.. controls (1.92,-.17) and (1.25,-.17) .. cycle;
\draw[outline]
(1.02,0) .. controls (1.25,-.17) and (1.92,-.17) .. (2.17,0);
\draw[hidden]
(1.02,0) .. controls (1.25,.17) and (1.92,.17) .. (2.17,0);

\draw[outline] (4.35,.35) arc[start angle=90,end angle=270,
x radius=.12,y radius=.35];
\draw[hidden] (4.35,-.35) arc[start angle=-90,end angle=90,
x radius=.12,y radius=.35];

\draw[outline] (5.20,.35) arc[start angle=90,end angle=270,
x radius=.11,y radius=.35];
\draw[hidden] (5.20,-.35) arc[start angle=-90,end angle=90,
x radius=.11,y radius=.35];%
}

\makeatletter
\@namedef{subjclassname@2020}{%
 \textup{2020} Mathematics Subject Classification}
\makeatother

\newtheorem{theorem}{Theorem}[section]
\newtheorem{proposition}[theorem]{Proposition}
\newtheorem{lemma}[theorem]{Lemma}
\newtheorem{corollary}[theorem]{Corollary}
\theoremstyle{remark}
\newtheorem{remark}[theorem]{Remark}

\newcommand{\B}{\mathbb B}
\newcommand{\R}{\mathbb R}
\newcommand{\RP}{\mathbb {RP}}
\newcommand{\CP}{\mathbb {CP}}
\newcommand{\HP}{\mathbb {HP}}
\newcommand{\Sym}{\operatorname{Sym}}
\newcommand{\sys}{\operatorname{sys}_1}
\newcommand{\Ric}{\operatorname{Ric}}
\newcommand{\Rm}{\operatorname{Rm}}

\newcommand{\interior}{\operatorname{int}}

\hypersetup{
 pdftitle={Small normal curvature and three-manifold topology},
 pdfauthor={Tsz-Kiu Aaron Chow and Jingbo Wan},
 pdfsubject={Sharp normal curvature bounds and three-manifold classification},
 pdfkeywords={normal curvature, systole, Ricci flow with surgery,
 real projective space, spherical space form}
}

\title{Small Normal Curvature and Three-Manifold Topology}
\author{Tsz-Kiu Aaron Chow}
\address{Department of Mathematics, Hong Kong University of Science and Technology, Hong Kong S.A.R., China}
\email{\href{chowtka@ust.hk}{chowtka@ust.hk}}
\author{Jingbo Wan}
\address{Laboratoire Jacques-Louis Lions de Sorbonne Universit\'e, 4 place Jussieu, Paris 75005, France}
\email{\href{jingbo.wan@sorbonne-universite.fr}{jingbo.wan@sorbonne-universite.fr}}
\date{}
\subjclass[2020]{53C42, 53E20, 57K30}
\keywords{normal curvature, systole, Ricci flow with surgery, real projective space, spherical space form}

\begin{document}

\begin{abstract}
For $m=2,3$, we prove that every smooth immersion $F:\RP^m\looparrowright\overline{\B}^{\,N}(1)$ satisfies $\kappa(F)^2\ge 2m/(m+1)$, with equality only for the Veronese embedding, up to congruence. We also prove that a closed, connected, orientable three-manifold admitting an immersion into a Euclidean unit ball with $\kappa(F)\le\sqrt{3/2}$ is diffeomorphic to $S^3$, $\RP^3$, or $S^2\times S^1$. All three possibilities occur, while $\kappa(F)<\sqrt{3/2}$ forces $X\cong S^3$. These results answer a question of Petrunin and prove a conjecture of Chodosh–Li concerning the normal curvature of three-manifolds. The key intrinsic input is the strict scalar--systolic inequality
\[
(\min_Y R_g)\sys(g)^2<6\pi^2
\]
for every spherical three-space form $Y$ with $|\pi_1(Y)|>2$. Its proof uses systolic monotonicity along Ricci flow with surgery. This strict inequality complements the sharp scalar--systolic inequality for $\RP^3$ of Bray--Brendle--Eichmair--Neves.
\end{abstract}

\maketitle

\section{Introduction}

Let $F:M^m\looparrowright\overline{\B}^{\,N}(1)\subset\R^N$ be a smooth immersion, let $g=F^*g_{\R^N}$, and let $A_F$ be its second fundamental form. We write
\[
 \kappa(F):=\sup_{p\in M}\sup_{|v|_g=1}|A_F(v,v)|.
\]
Minimizing this quantity for a prescribed topology belongs to Gromov's program on immersions with controlled curvatures \cite{GromovDesign,Gromov,GromovLectures}. For real projective space, the natural model is the Veronese embedding
\begin{equation}\label{eq:veronese}
 V_m:\RP^m\longrightarrow S\bigl(\Sym_0(m+1)\bigr),\quad
 V_m([q])=\sqrt{\frac{m+1}{m}}
 \left(qq^{\mathsf T}-\frac1{m+1}I_{m+1}\right),\quad q\in S^m.
\end{equation}
It satisfies $\kappa(V_m)^2=2m/(m+1)$. Petrunin proved its optimality for $\RP^2$ and asked whether the same holds for every $\RP^m$ \cite{PetruninVeronese,Petrunin}. We prove the first open case, $m=3$.

\begin{theorem}\label{thm:main}
For $m\in\{2,3\}$, every smooth immersion $F:(\RP^m,g)\looparrowright\overline{\B}^{\,N}(1)$ satisfies
\begin{equation}\label{eq:main}
 \kappa(F)^2\geq\frac{2m}{m+1}.
\end{equation}
Equality holds precisely for the Veronese embedding in \eqref{eq:veronese}, up to congruence.
% \footnote{more precisely, up to an isometry of the domain, an orthogonal transformation of the target, and a linear isometric inclusion $\Sym_0(m+1)\hookrightarrow\R^N$}
\end{theorem}

We write $R$ for scalar curvature. For a closed Riemannian manifold $(M,h)$ with $\pi_1(M)\ne1$, let $\sys(M,h)$ denote the least length of a noncontractible closed curve; we write $\sys(h)$ when $M$ is understood.

The proof of Theorem~\ref{thm:main} reduces normal curvature to an intrinsic scalar--systolic inequality. For a closed immersed $M^m$ with induced metric $g$ and $\pi_1(M)\ne1$, set
\[
 \widetilde g=\exp\left(-\frac{3m}{2(m-1)}|F|^2\right)g.
\]
If $\kappa(F)^2\le2m/(m+1)$, the conformal scalar curvature formula and Fenchel's theorem give
\[
 \bigl(\inf_MR_{\widetilde g}\bigr)\sys(\widetilde g)^2
 \ge m(m-1)\pi^2,
\]
strictly when $\kappa(F)^2<2m/(m+1)$. For $\RP^2$, this systolic inequality follows from Gauss--Bonnet and Pu's theorem \cite{Pu}; for $\RP^3$, it follows from the Bray--Brendle--Eichmair--Neves theorem \cite{BBEN}. This proves the inequality in Theorem~\ref{thm:main}. At equality, $\widetilde g$ is round and $F$ is planar-geodesic; the classification of Little and Sakamoto then gives the Veronese embedding \cite{Little,Sakamoto}.

We next turn to closed orientable three-manifolds. We improve the threshold $\sqrt{4/3}$ in Chodosh--Li's differentiable sphere theorem \cite[Theorem~5]{ChodoshLi} to the sharp value $\sqrt{3/2}$ and prove their conjectured classification at the threshold \cite[Remark~3]{ChodoshLi}.

\begin{theorem}\label{thm:classification}
Let $X^3$ be closed, connected, and orientable. If, for some $N$, there is a smooth immersion $F:X\looparrowright\overline{\B}^{\,N}(1)$ with $\kappa(F)\le \sqrt{3/2}$, then $X\cong S^3$, $X\cong\RP^3$, or $X\cong S^2\times S^1$. All three possibilities occur. If $\kappa(F)<\sqrt{3/2}$, then $X\cong S^3$.
\end{theorem}

The standard inclusion $S^3\hookrightarrow\overline{\B}^{\,4}(1)$ has $\kappa=1$, the Veronese embedding $V_3$ in \eqref{eq:veronese} realizes $\RP^3$ with $\kappa=\sqrt{3/2}$, and the immersion in \cite[(3)]{ChodoshLi} realizes $S^2\times S^1$ with $\kappa=\sqrt{3/2}$. The orientability hypothesis in Theorem~\ref{thm:classification} is essential: the immersion in \cite[(3)]{ChodoshLi} is invariant under the free orientation-reversing involution $(x,t)\longmapsto(-x,t+\pi)$, and hence descends, with the same maximal normal curvature, to the quotient, the twisted $S^2$-bundle over $S^1$.

The key new intrinsic input is the following systolic strict gap for spherical space forms.

\begin{theorem}\label{thm:gap}
Let $Y=S^3/G$ be a spherical three-space form with $|G|>2$. Every smooth Riemannian metric $g$ on $Y$ satisfies
\[
 (\min_YR_g)\sys(g)^2<6\pi^2.
\]
\end{theorem}

Combining Theorem~\ref{thm:gap} with the sharp $\RP^3$ scalar--systolic inequality \cite{BBEN} gives the corresponding result for every non-simply connected spherical space form.

\begin{corollary}\label{cor:spherical}
If $Y$ is a non-simply connected spherical three-space form, then
\[
 (\min_YR_g)\sys(g)^2\le6\pi^2.
\]
Equality holds if and only if $Y\cong\RP^3$ and $g$ is round, up to scaling and diffeomorphism.
\end{corollary}

The same conformal estimate gives the following corollary.

\begin{corollary}\label{cor:spherical-immersion}
Let $Y=S^3/G$ be a non-simply connected spherical three-space form. Every smooth immersion $F:Y\looparrowright\overline{\B}^{\,N}(1)$ satisfies $\kappa(F)^2\geq 3/2$. Equality holds precisely when $Y\cong\RP^3$ and, after identifying $Y$ with $\RP^3$, $F$ is the Veronese embedding $V_3$ in \eqref{eq:veronese}, up to congruence. In particular, the inequality is strict if $|G|>2$.
\end{corollary}

\begin{remark}
In view of the strict gap in Theorem~\ref{thm:gap}, it would be interesting to determine the sharp scalar--systolic constant for lens spaces. For the homogeneous lens spaces $L(p,1)$, $p>2$, we suspect the extremal metric to be the Berger metric $g_B$ obtained from the round metric by scaling the Hopf-fiber direction by a factor $\sqrt2$. Thus we suspect that
\[
(\min_{L(p,1)}R_g)\sys(g)^2
\le
\frac{32\pi^2}{p^2},
\]
with equality if and only if $g$ is $g_B$, up to scaling and diffeomorphism.
\end{remark}

To prove Theorem~\ref{thm:classification}, the argument of \cite[Proposition~2]{ChodoshLi} produces a conformal metric with nonnegative sectional curvature, positive when $\kappa(F)^2<3/2$. Hamilton's classification \cite[Theorem~1.2]{Hamilton1986} leaves spherical space forms, compact quotients of $S^2\times\R$, and flat manifolds. The scalar--systolic estimate excludes the flat case, the projective-plane estimate \cite{BBEN} excludes $\RP^3\mathbin{\#}\RP^3$, and Theorem~\ref{thm:gap} excludes spherical quotients of order greater than two. This proves the classification. Below the threshold, positive sectional curvature and the strict scalar--systolic estimate force $X\cong S^3$.

The proof of Theorem~\ref{thm:gap} uses Ricci flow with surgery in dimension three. Hamilton's surgery program was carried out in dimension three by Perelman, with detailed treatments by Cao--Zhu, Morgan--Tian, and Kleiner--Lott \cite{Hamilton1995,Perelman2003,CaoZhu2006,MorganTian2007,KleinerLott2008}.  The proof of Theorem~\ref{thm:gap} uses this framework in a simple geometric way. On each smooth interval, Proposition~\ref{prop:smooth-systole} gives $(\sys^2)'\ge-4\pi^2$. At surgery, any portion of an essential loop entering a new cap can be replaced by a shorter path on an outer neck cross-section, so Proposition~\ref{prop:surgery-systole} prevents the systole from jumping down. Thus $\sys(g(t))^2+4\pi^2t$ is nondecreasing along the smooth flow and across surgery times. Together with the scalar curvature maximum principle, this carries the scalar--systolic lower bound to the positive sectional curvature slice supplied by Ricci flow with surgery; see Proposition~\ref{prop:ricci-input}(vi). Hamilton's theorem then gives a round limit \cite{Hamilton1982}; since every round spherical space form of order greater than two has systole strictly smaller than round $\RP^3$, the gap follows.

Our Ricci flow viewpoint is partly motivated by the rigidity argument of Bray--Brendle--Eichmair--Neves \cite{BBEN} and by Chodosh--Li's use of the Brendle--Schoen sphere theorem to derive a differentiable sphere theorem from a normal curvature bound \cite{ChodoshLi,BrendleSchoen2009}. The systolic part of the argument is also in the spirit of monotonicity methods in geometric flows with surgery: Perelman used a min--max least-area disk functional and showed that its differential inequality persists through Ricci flow surgery \cite{PerelmanExtinction}, while Colding--Minicozzi used the min--max width of two-sphere sweepouts to prove finite extinction \cite{ColdingMinicozziExtinction}; in mean curvature flow with surgery, Brendle established a modified Huisken-type monotonicity formula that remains monotone across surgery times \cite{BrendleSurgeryMonotonicity}. The conformal construction is related to work of Petrunin on tori \cite{PetruninTori} and of Chodosh--Li on $S^n\times S^1$ \cite{ChodoshLi}. Related normal curvature estimates were obtained by Raffaelli \cite{RaffaelliDomains,RaffaelliTori}; see also our earlier Veronese rigidity results for $\CP^m,\HP^m$ under additional intrinsic geometric assumptions \cite{ChowWan}.

\medskip

\textit{Acknowledgements}
T.-K. A. C. was supported by the Croucher Foundation and HKUST New Faculty start-up grants. J. W. was supported by ERC-2023-AdG 101141855 BLaHST. Part of this work was carried out during a visit by J. W. to the HKUST Department of Mathematics, which he thanks for its hospitality.

\medskip

\textit{Disclosure of AI use.} The authors used ChatGPT (GPT-5.6 Sol) and Claude (Fable 5) solely for supporting tasks, including literature searches, exploration of ideas, and symbolic calculations. The authors independently verified all calculations and proofs and take full responsibility for the mathematical content.

\section{The normal curvature--systolic bridge}
\label{sec:bridge}

The first ingredient is a bow type estimate.

\begin{lemma}\label{lem:bow}
For every closed immersed manifold in the unit ball, $\kappa\geq1$.
If $\pi_1(M)\ne1$, then
\begin{equation}\label{eq:bow-sys}
 \sys(g)\geq\frac{2\pi}{\kappa}.
\end{equation}
If $\kappa<2$, then
\begin{equation}\label{eq:radial-bow}
 |F^\perp|\geq1+\frac{\kappa}{2}(|F|^2-1),
 \qquad
 \left(\frac2\kappa-1\right)^2\leq|F|^2\leq1.
\end{equation}
\end{lemma}

\begin{proof}
The identities $\Delta|F|^2/2=m+\langle F,H\rangle$ and $|H|\leq m\kappa$ give
\[
 0=\int_M\Delta\frac{|F|^2}{2}
 =m\operatorname{vol}_g(M)+\int_M\langle F,H\rangle
 \geq m(1-\kappa)\operatorname{vol}_g(M),
\]
hence $\kappa\geq1$. If $\pi_1(M)\ne1$, let $\gamma$ be a shortest noncontractible closed geodesic and set $x=F\circ\gamma$. Then $|x'|=1$ and $x''=A_F(\gamma',\gamma')$. Fenchel's theorem \cite{Milnor} implies \eqref{eq:bow-sys}, since
\[
 2\pi\leq\int_0^{\sys(g)}|x''(t)|\,dt
 \leq\kappa\sys(g).
\]
Next, we note that the first estimate in \eqref{eq:radial-bow} follows from \cite[Lemma~12]{ChodoshLi}. Since $|F^\perp|\leq|F|\leq1$, the first estimate gives
$|F|\geq1+\frac{\kappa}{2}(|F|^2-1)$, hence $0\geq(|F|-1)\left(\frac{\kappa}{2}(|F|+1)-1\right)$. Since $|F|\leq1$, we obtain $|F|\geq2/\kappa-1$, and therefore $\left(2/\kappa-1\right)^2\leq|F|^2\leq1$. This proves \eqref{eq:radial-bow}.
\end{proof}

To derive the conformal scalar curvature estimate, we first record a
calculus lemma.

\begin{lemma}\label{lem:conformal-calculus}
Let $m\geq2$ and $1\leq \kappa^2\leq\frac{2m}{m+1}$. Then for $0\leq y\leq\frac{4(\kappa-1)}{\kappa^2}$,
\begin{align}\label{eq:conformal-calculus}
\frac m2\bigl(3m-(m+2)\kappa^2\bigr) +\frac{3m^2}{2}y -\frac{3m^2(m-4)_+}{4(m-1)} \left( (\kappa-1)y-\frac{\kappa^2}{4}y^2 \right) 
\geq 
\frac{m(m-1)}4\kappa^2 e^{\frac{3m}{2(m-1)}y}.
\end{align}
Equality holds only if $\kappa^2=\frac{2m}{m+1}$ and $y=0$.
\end{lemma}

\begin{proof}
Set $a=\frac{3m}{2(m-1)}$, $b=\frac{3m^2(m-4)_+}{4(m-1)}$ and
$D=\frac{m(m-1)}4$. Denote the left hand side in
\eqref{eq:conformal-calculus} by
\[
P(y) = \frac m2\bigl(3m-(m+2)\kappa^2\bigr) +\frac{3m^2}{2}y -b\left((\kappa-1)y-\frac{\kappa^2}{4}y^2\right).
\]
Now we prove $P(y)\geq D\kappa^2e^{ay}$. Define $f(y)=P(y)-D\kappa^2e^{ay}$. Since $P''(y)=\frac{b\kappa^2}{2}$, we have
\begin{align*}
f''(y)= \frac{b\kappa^2}{2}-a^2D\kappa^2e^{ay}\leq \kappa^2\left(\frac{b}{2} - a^2 D\right)= \frac{3m^2\kappa^2}{16(m-1)} \bigl(2(m-4)_+-3m\bigr)<0 .
\end{align*}
Hence $f$ is strictly concave, and its minimum is attained at an endpoint.

At $y=0$,
\begin{align*}
f(0) = \frac m2\bigl(3m-(m+2)\kappa^2\bigr)-D\kappa^2 = \frac{3m}{4}\bigl(2m-(m+1)\kappa^2\bigr)\geq0 .
\end{align*}

At the other endpoint $y_0=\frac{4(\kappa-1)}{\kappa^2}$, the quadratic term vanishes since $(\kappa-1)y_0-\frac{\kappa^2}{4}y_0^2=0$. Writing $z_0=ay_0=\frac{6m(\kappa-1)}{(m-1)\kappa^2}$, we compute using $2m-(m+1)\kappa^2\geq0$
\begin{align*}
f(y_0)
&= \frac m2\bigl(3m-(m+2)\kappa^2\bigr) +\frac{3m^2}{2}y_0 -D\kappa^2e^{z_0}\\
&= D\kappa^2 \left( 1+ \frac{3(2m-(m+1)\kappa^2)} {(m-1)\kappa^2} +\frac{4z_0}{\kappa^2} -e^{z_0} \right)\\
&\geq D\kappa^2 \left( 1+\frac{4z_0}{\kappa^2}-e^{z_0} \right).
\end{align*}
We claim that $f(y_0)\geq 0$. A direct calculation shows that for $1\leq\kappa\leq\sqrt{2}$, one has
\begin{equation}\label{eq:exp-kappa}
e^{\frac{3}{1+\kappa}}
\leq
1+\frac{12}{(1+\kappa)\kappa^2}.
\end{equation}
On the other hand, the assumption $\kappa^2\leq\frac{2m}{m+1}$ implies $z_0 =\frac{6m(\kappa-1)}{(m-1)\kappa^2} \leq \frac{3}{1+\kappa}$. Putting this into \eqref{eq:exp-kappa} gives $e^{z_0}\leq1+\frac{4z_0}{\kappa^2}$, and hence $f(y_0)\geq 0$. The strict concavity of $f$ now gives $f(y)\geq0$ on the whole interval. The equality statement follows from the strict inequality at the right endpoint and the equality case at $y=0$.
\end{proof}

The following conformal change is analogous to Petrunin's work \cite{PetruninTori} on $T^n$ and Chodosh--Li's work \cite{ChodoshLi} on $S^n\times S^1$.

\begin{proposition}\label{prop:bridge}
Let $\widetilde g= \exp\left(-\frac{3m}{2(m-1)}|F|^2\right)g$. If $\kappa^2\leq\frac{2m}{m+1}$, then
\begin{equation}\label{eq:conformal}
 R_{\widetilde g} \geq \exp\left(\frac{3m}{2(m-1)}\right) \frac{m(m-1)}4\kappa^2 .
\end{equation}
Equality at a point in \eqref{eq:conformal} implies $\kappa^2=\frac{2m}{m+1}$ and $|F|^2=1$ there. If $\pi_1(M)\ne1$, then
\begin{equation}\label{eq:bridge}
 \bigl(\inf_MR_{\widetilde g}\bigr)\sys(\widetilde g)^2
 \ge m(m-1)\pi^2,
\end{equation}
strictly when $\kappa^2<2m/(m+1)$.
\end{proposition}

\begin{proof}
First, the fourth moment computation on unit tangent sphere and the Gauss equation imply average Gauss equation; see \cite[2.1. Gauss formula]{PetruninTori}
\begin{equation*}
 R_g=\frac32|H|^2-\frac{m(m+2)}2\text{\CYRZH},\qquad
 \text{\CYRZH}:=\fint_{S_pM}|A_F(v,v)|^2\,dv .
\end{equation*}

Set $u=-\frac{3m}{4(m-1)}|F|^2$, so that $\widetilde g=e^{2u}g$. The conformal change formula gives
\[
R_{\widetilde g} = e^{-2u} \left( R_g-2(m-1)\Delta u-(m-1)(m-2)|\nabla u|^2 \right).
\]
Using $\Delta|F|^2=2m+2\langle F^\perp,H\rangle$ and $|\nabla|F|^2|^2=4|F^\top|^2$, we obtain
\begin{align*}
e^{2u}R_{\widetilde g}
&= R_g +\frac{3m}{2}\Delta|F|^2 -\frac{9m^2(m-2)}{16(m-1)} |\nabla|F|^2|^2\\
&= 3m^2-\frac{m(m+2)}2 \text{\CYRZH} +\frac32|H|^2 +3m\langle F^\perp,H\rangle -\frac{9m^2(m-2)}{4(m-1)} |F^\top|^2 .
\end{align*}
Completing the square in the mean curvature term gives
\[
\frac32|H|^2+3m\langle F^\perp,H\rangle = \frac32|H+mF^\perp|^2 -\frac{3m^2}{2}|F^\perp|^2 .
\]
Therefore,
\begin{align}
e^{2u}R_{\widetilde g}
&= 3m^2-\frac{m(m+2)}2 \text{\CYRZH} \notag +\frac32|H+mF^\perp|^2 -\frac{3m^2}{2}|F^\perp|^2 -\frac{9m^2(m-2)}{4(m-1)} |F^\top|^2 \notag\\
&= 3m^2-\frac{m(m+2)}2 \text{\CYRZH} -\frac{3m^2}{2}|F|^2 +\frac32|H+mF^\perp|^2 +\frac{3m^2(4-m)}{4(m-1)}|F^\top|^2 \notag\\
&\geq 3m^2-\frac{m(m+2)}2\kappa^2 -\frac{3m^2}{2}|F|^2 +\frac{3m^2(4-m)}{4(m-1)}|F^\top|^2 . \label{eq:defect}
\end{align}

Set $y=1-|F|^2$. Lemma~\ref{lem:bow} gives $0\leq y\leq 1-\left(\frac2\kappa-1\right)^2= \frac{4(\kappa-1)}{\kappa^2}$ and $|F^\perp|\geq1-\frac{\kappa}{2}y$. Furthermore,
\begin{align*}
|F^\top|^2=
1-y-|F^\perp|^2\leq
1-y-\left(1-\frac{\kappa}{2}y\right)^2
=(\kappa-1)y-\frac{\kappa^2}{4}y^2 .
\end{align*}
If $m\leq4$, the coefficient of $|F^\top|^2$ is nonnegative and the last term can be discarded. If $m>4$, then the coefficient is negative, and the above estimate gives
\[
\frac{3m^2(4-m)}{4(m-1)}|F^\top|^2
\geq
-\frac{3m^2(m-4)}{4(m-1)}
\left(
(\kappa-1)y-\frac{\kappa^2}{4}y^2
\right).
\]
Thus in all cases, use Lemma~\ref{lem:conformal-calculus}
\begin{align*}
e^{2u}R_{\widetilde g}
&\geq
\frac m2\bigl(3m-(m+2)\kappa^2\bigr)
+\frac{3m^2}{2}y
-\frac{3m^2(m-4)_+}{4(m-1)}
\left((\kappa-1)y-\frac{\kappa^2}{4}y^2\right)\\
&\geq
e^{\frac{3m}{2(m-1)}y}
\frac{m(m-1)}4\kappa^2 .
\end{align*}
Finally, multiplying by $e^{-2u}=\exp(\frac{3m}{2(m-1)}|F|^2)$ gives \eqref{eq:conformal}. The equality assertion follows from the equality assertion in Lemma~\ref{lem:conformal-calculus}.

If $\pi_1(M)\ne1$, then $\widetilde g\ge \exp\left(-\frac{3m}{2(m-1)}\right)g$ and Lemma~\ref{lem:bow} give
\[
 \sys(\widetilde g)^2
 \ge \exp\left(-\frac{3m}{2(m-1)}\right)\sys(g)^2
 \ge \exp\left(-\frac{3m}{2(m-1)}\right)
 \frac{4\pi^2}{\kappa^2}.
\]
Multiplying this with \eqref{eq:conformal} proves \eqref{eq:bridge}. If $\kappa^2<2m/(m+1)$, the pointwise estimate \eqref{eq:conformal} is strict; compactness makes the infimum strict as well.
\end{proof}

Next, we record the sharp scalar–systolic inequalities.
\begin{proposition}\label{prop:systolic}\

\begin{enumerate}[label=(\roman*)]
\item Every metric $h$ on $\RP^2$ satisfies
\[
 (\inf_{\RP^2}R_h)\sys(h)^2\le2\pi^2.
\]
Equality holds only for a round metric, up to scaling and diffeomorphism.
\item If a closed three-manifold $(M,h)$ contains a smoothly embedded $\RP^2$, then
\[
 (\inf_MR_h)\sys(M,h)^2\le6\pi^2.
\]
Equality holds only when $(M,h)$ is a round $\RP^3$, up to scaling and diffeomorphism.
\end{enumerate}
\end{proposition}

\begin{proof}
Both assertions are immediate if $\inf R_h \le 0$. For (i), multiplying Gauss--Bonnet, $(\inf R_h)\operatorname{area}_h(\RP^2)\le 4\pi$, with Pu's inequality \cite{Pu}, $\sys(h)^2\le \frac{\pi}{2}\operatorname{area}_h(\RP^2)$, gives the bound; equality forces $h$ to be round. For (ii), let $\mathcal A(M,h)$ be the least area of an embedded projective plane. Multiplying the two inequalities $\frac2\pi\sys(M,h)^2 \le \mathcal A(M,h)$ and $\mathcal A(M,h)\inf_M R_h \le 12\pi$ from \cite[Theorem~1]{BBEN} yields the conclusion, with rigidity given by \cite[Theorem~2]{BBEN}.
\end{proof}

\section{Systolic estimates via Ricci flow with surgery}\label{sec:ricci-input}

\subsection{Geometric setup and notation}\label{sec:setup}

Let $Y=S^3/G$ be a spherical space form with $|G|>2$, and a smooth Riemannian metric $k$. The manifold $Y$ is orientable and irreducible.

We use the following notation.
\begin{itemize}[leftmargin=2em]
\item \emph{Flow and singular times.}
We write $g(0)=k$ and use $g(t)$ for a three-dimensional Ricci flow with $(r,\delta)$-cutoff in the sense of \cite[Definition~73.1]{KleinerLott2008}. Set $t_0=0$, and write $0<t_1<t_2<\cdots$ for the singular times, indexed by $1\le j\le j_{\max} \in\mathbb N_0\cup\{\infty\}$. If $j_{\max}<\infty$, set $t_{j_{\max}+1}=\infty$. Let $M_0=Y$, and for $1\le j\le j_{\max}$ let $M_j$ be the postsurgery manifold at $t_j$. 
At a singular time $T=t_j$, the open set on which the curvature remains bounded as $t\uparrow T$ is denoted by $\Omega(T)$. The metrics converge smoothly on compact subsets of $\Omega(T)$ to $g(T^-)$, and we write $\mathcal M_T^-=(\Omega(T),g(T^-))$, $\mathcal M_T^+=(M_j,g(T^+))$.

\item \emph{Scalar quantities.}
Let $\mu_0:=\min_YR_k$. When $\mu_0>0$, set
\[
 b(t):=\frac{\mu_0}{1-\frac23\mu_0t},
 \qquad
 T_0:=\frac3{2\mu_0}.
\]

\item \emph{The $Y$-component and its systole.}
As long as a component diffeomorphic to $Y$ occurs, it is denoted by $Y(t)$ at a nonsingular time and by $Y(T^+)$ after surgery. We set
\[
 s(t)=\sys(Y(t),g(t)),
 \qquad
 s(T^+)=\sys(Y(T^+),g(T^+)),
 \qquad
 s_0=\sys(k).
\]

\item \emph{Surgery scale and cylindrical model.}
At a singular time $T$, we write $\delta=\delta(T)$, $\rho=\delta(T)r(T)$, and $h=h(T)$ for the surgery scale. In the notation of \cite[(73.2)]{KleinerLott2008},
\[
 \Omega_\rho=\{x\in\Omega(T):R(x,T^-)\le\rho^{-2}\}.
\]
The cylindrical model used in \cite[Definition~58.1]{KleinerLott2008} is $g_{\rm cyl}=2g_{S^2(1)}+dz^2$, which has scalar curvature one and cross-sectional diameter $\pi\sqrt2$.

\item \emph{Neck and cap regions.}
We fix $L>\pi/\sqrt2$ and write $\eta=\eta(\delta)$ for the error in \eqref{eq:omega} below. In the coordinates of \cite[Lemma~72.24]{KleinerLott2008}, surgery leaves $[0,\delta^{-1})\times S^2$ unchanged, and $D_i$ is the added three-disk bounded by $\{0\}\times S^2$. Set
\[
 B_i=D_i\cup\bigl([0,L]\times S^2\bigr),
 \qquad
 \Sigma_i=\partial B_i=\{L\}\times S^2 .
\]

\end{itemize}

\subsection{Ricci flow with surgery}

The next proposition collects the Ricci flow results used below.

\begin{proposition}\label{prop:ricci-input}
Assume that $k$ is normalized in the sense of \cite[Definition~77.1]{KleinerLott2008}. There are positive nonincreasing functions $r$ and $\delta$ and a Ricci flow with $(r,\delta)$-cutoff starting from $(Y,k)$ with the following properties.
\begin{enumerate}[label=(\roman*)]
\item The flow is defined for all $t\ge0$, with the time slices taken to be empty after extinction. Every bounded time interval contains only finitely many singular times.

\item At a singular time $T$, the metrics converge smoothly to $g(T^-)$ on compact subsets of $\Omega(T)$. Under the surgery identification,
\[
 X:=\mathcal M_T^+\setminus\bigcup_i\interior D_i
 \Subset\Omega(T),
 \qquad g(T^+)|_X=g(T^-)|_X.
\]

\item If $Y(T^+)$ exists, there are finitely many pairwise disjoint three-balls $B_i\subset Y(T^+)$ with $D_i\subset\interior B_i$, and
\[
 U:=Y(T^+)\setminus\bigcup_i\interior D_i
\]
is connected and compactly contained in $\Omega(T)$. Thus $U=X\cap Y(T^+)$ and $U\cap B_i=B_i\setminus\interior D_i=[0,L]\times S^2$. If $\Sigma_i=\partial B_i$, then
\begin{equation}\label{eq:neck-lengths}
 \operatorname{diam}_{(\Sigma_i,g(T^+)|_{\Sigma_i})}(\Sigma_i)
 \le(\pi\sqrt2+\eta)h,
 \qquad
 \inf_\sigma L_{g(T^+)}(\sigma)\ge(L-\eta)h,
\end{equation}
where the infimum is over paths $\sigma\subset B_i\setminus\interior D_i$ joining $\partial D_i$ to $\Sigma_i$, and $\eta=\eta(\delta)\to0$. Moreover,
\begin{equation}\label{eq:neck-inequality}
 2(L-\eta)>\pi\sqrt2+\eta.
\end{equation}

\item For some $0\le j_0\le j_{\max}$, each $M_j$ ($0\le j\le j_0$) has a unique component $M_j^{(0)}\cong Y$ with all other components diffeomorphic to $S^3$, while for $j_0<j\le j_{\max}$ all components are diffeomorphic to $S^3$. If $j_0<j_{\max}$, $M_{j_0}^{(0)}$ goes extinct at $t_{j_0+1}$.

\item If $\mu_0>0$, then
\[
 R(\cdot,t)\ge b(t)\quad(t<T_0),
 \qquad R(\cdot,T^+)\ge b(T)\quad(T<T_0),
\]
at every nonsingular time $t$ and every singular time $T$, respectively. The flow is empty for $t\ge T_0$.

\item If $\mu_0>0$, there is a nonsingular time $t_*<T_0$, before the component diffeomorphic to $Y$ ceases to occur, such that $\sec_{g(t_*)}>0$ on $Y(t_*)$.
\end{enumerate}
\end{proposition}

% -----------------------------------------------------------------------------
% Figure 1: the geometric surgery operation
% -----------------------------------------------------------------------------
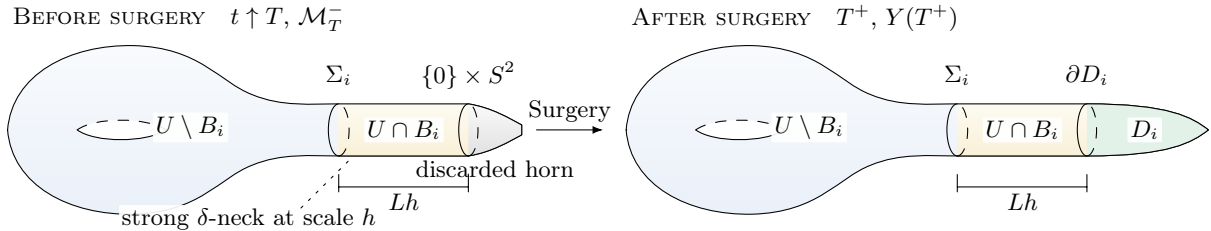
\begin{figure}[ht]
\centering
\begin{tikzpicture}[x=0.98cm,y=0.82cm,font=\scriptsize]

% Presurgery component.
\begin{scope}
\draw[body]
(.08,0)
.. controls (.14,.90) and (.92,1.38) .. (1.78,1.36)
.. controls (2.68,1.33) and (2.94,.77) .. (3.38,.57)
.. controls (3.78,.37) and (4.18,.42) .. (4.55,.42)
-- (6.30,.42)
.. controls (6.62,.36) and (6.86,.18) .. (7.02,.08)
-- (7.02,-.08)
.. controls (6.86,-.18) and (6.62,-.36) .. (6.30,-.42)
-- (4.55,-.42)
.. controls (4.18,-.42) and (3.78,-.37) .. (3.38,-.57)
.. controls (2.94,-.77) and (2.68,-1.33) .. (1.78,-1.36)
.. controls (.92,-1.38) and (.14,-.90) .. (.08,0);

\path[collar region]
(4.55,.42) -- (6.30,.42) -- (6.30,-.42) -- (4.55,-.42) -- cycle;

\path[shade,top color=black!2,bottom color=black!12]
(6.30,.42)
.. controls (6.62,.36) and (6.86,.18) .. (7.02,.08)
-- (7.02,-.08)
.. controls (6.86,-.18) and (6.62,-.36) .. (6.30,-.42)
-- cycle;

\draw[outline]
(4.55,.42) -- (6.30,.42)
.. controls (6.62,.36) and (6.86,.18) .. (7.02,.08)
-- (7.02,-.08)
.. controls (6.86,-.18) and (6.62,-.36) .. (6.30,-.42)
-- (4.55,-.42);

\path[fill=white]
(1.02,0)
.. controls (1.26,.21) and (1.95,.21) .. (2.22,0)
.. controls (1.95,-.21) and (1.26,-.21) .. cycle;
\draw[outline]
(1.02,0)
.. controls (1.26,-.21) and (1.95,-.21) .. (2.22,0);
\draw[hidden]
(1.02,0)
.. controls (1.26,.21) and (1.95,.21) .. (2.22,0);

\draw[outline]
(4.55,.42) arc[
  start angle=90,end angle=270,
  x radius=.14,y radius=.42
];
\draw[hidden]
(4.55,-.42) arc[
  start angle=-90,end angle=90,
  x radius=.14,y radius=.42
];

\draw[outline]
(6.30,.42) arc[
  start angle=90,end angle=270,
  x radius=.13,y radius=.42
];
\draw[hidden]
(6.30,-.42) arc[
  start angle=-90,end angle=90,
  x radius=.13,y radius=.42
];

\node[clear label,anchor=south west] at (.10,1.54)
  {\textsc{Before surgery}\quad $t\uparrow T$, $\mathcal M^-_T$};
\node[clear label] at (2.55,.04) {$U\setminus B_i$};
\node[clear label] at (4.55,.88) {$\Sigma_i$};
\node[clear label] at (6.30,.88) {$\{0\}\times S^2$};
\node[clear label] at (5.43,0) {$U\cap B_i$};
\node[clear label] at (6.66,-.62) {discarded horn};

\draw[outline]
(4.55,-.72) -- (4.55,-1.00)
(6.30,-.72) -- (6.30,-1.00)
(4.55,-.92) -- (6.30,-.92);
\node[clear label] at (5.43,-1.16) {$Lh$};

\node[clear label] at (3.34,-1.43)
  {strong $\delta$-neck at scale $h$};
\draw[leader] (4.04,-1.26) -- (4.75,-.42);
\end{scope}

\draw[process] (7.22,0) -- (8.12,0);
\node[clear label] at (7.67,.31) {Surgery};

% Postsurgery component.
\begin{scope}[shift={(8.35,0)}]
\draw[body]
(.08,0)
.. controls (.14,.90) and (.92,1.38) .. (1.78,1.36)
.. controls (2.68,1.33) and (2.94,.77) .. (3.38,.57)
.. controls (3.78,.37) and (4.18,.42) .. (4.55,.42)
-- (6.30,.42)
.. controls (7.42,.42) and (7.72,.20) .. (7.94,0)
.. controls (7.72,-.20) and (7.42,-.42) .. (6.30,-.42)
-- (4.55,-.42)
.. controls (4.18,-.42) and (3.78,-.37) .. (3.38,-.57)
.. controls (2.94,-.77) and (2.68,-1.33) .. (1.78,-1.36)
.. controls (.92,-1.38) and (.14,-.90) .. (.08,0);

\path[fill=white]
(1.02,0)
.. controls (1.26,.21) and (1.95,.21) .. (2.22,0)
.. controls (1.95,-.21) and (1.26,-.21) .. cycle;
\draw[outline]
(1.02,0)
.. controls (1.26,-.21) and (1.95,-.21) .. (2.22,0);
\draw[hidden]
(1.02,0)
.. controls (1.26,.21) and (1.95,.21) .. (2.22,0);

\path[collar region]
(4.55,.42) -- (6.30,.42) -- (6.30,-.42) -- (4.55,-.42) -- cycle;

\path[fill=rfgreen]
(6.30,.42)
.. controls (7.42,.42) and (7.72,.20) .. (7.94,0)
.. controls (7.72,-.20) and (7.42,-.42) .. (6.30,-.42)
-- cycle;

\draw[outline]
(4.55,.42) -- (6.30,.42)
.. controls (7.42,.42) and (7.72,.20) .. (7.94,0)
.. controls (7.72,-.20) and (7.42,-.42) .. (6.30,-.42)
-- (4.55,-.42);

\draw[outline]
(4.55,.42) arc[
  start angle=90,end angle=270,
  x radius=.14,y radius=.42
];
\draw[hidden]
(4.55,-.42) arc[
  start angle=-90,end angle=90,
  x radius=.14,y radius=.42
];

\draw[outline]
(6.30,.42) arc[
  start angle=90,end angle=270,
  x radius=.13,y radius=.42
];
\draw[hidden]
(6.30,-.42) arc[
  start angle=-90,end angle=90,
  x radius=.13,y radius=.42
];

\node[clear label,anchor=south west] at (.10,1.54)
  {\textsc{After surgery}\quad $T^+$, $Y(T^+)$};
\node[clear label] at (2.55,.04) {$U\setminus B_i$};
\node[clear label] at (4.55,.88) {$\Sigma_i$};
\node[clear label] at (6.30,.88) {$\partial D_i$};
\node[clear label] at (5.43,0) {$U\cap B_i$};
\node[clear label] at (7.08,0) {$D_i$};

\draw[outline]
(4.55,-.72) -- (4.55,-1.00)
(6.30,-.72) -- (6.30,-1.00)
(4.55,-.92) -- (6.30,-.92);
\node[clear label] at (5.43,-1.16) {$Lh$};
\end{scope}

\end{tikzpicture}
\caption{\footnotesize Ricci flow surgery near one cap. The collar $U\cap B_i=B_i\setminus\interior D_i=[0,L]\times S^2$ is unchanged. Surgery removes the horn beyond $\{0\}\times S^2$ and attaches the three-disk $D_i$ there; $Lh$ denotes the model length of the collar.}
\label{fig:ricci-surgery}
\end{figure}

\begin{proof}
We apply the existence result of \cite[Section~5.1]{DinkelbachLeeb2009} in the nonequivariant case. Fix $0<\varepsilon_1\le\varepsilon_1^{(5)}$, where $\varepsilon_1$ measures the quality of approximation of canonical neighborhoods by local models. The existence statement there, based on \cite[Proposition~77.2]{KleinerLott2008}, provides positive nonincreasing functions $r$ and $\bar\delta$ such that every positive nonincreasing function $\delta<\bar\delta$ determines a Ricci flow with $(r,\delta)$-cutoff defined for all time. Decrease $\delta$ further if necessary for the estimates below. \cite[Lemma~73.7]{KleinerLott2008} excludes accumulation of singular times on bounded intervals, proving (i).

The convergence in (ii) follows from \cite[Section~67]{KleinerLott2008}. The common region $X$ and the equality of the two metrics there are part of the surgery identification in \cite[Definition~68.1]{KleinerLott2008} and the paragraph following \cite[Definition~73.1]{KleinerLott2008}.

To obtain (iii), for the $i$-th cap choose $x_i^0\in\{0\}\times S^2$ in the neck parametrization of \cite[Lemma~72.24]{KleinerLott2008} and set $Q_i=R(x_i^0,T^-)$. The neck comparison extends to the cut sphere, which is chosen through a point $x_i$ satisfying $R(x_i,T^-)=h^{-2}$ by \cite[Lemma~71.1 and Definition~73.1]{KleinerLott2008}. Since $x_i$ and $x_i^0$ lie a uniformly bounded distance apart in the same strong $\delta$-neck, the neck comparison gives
\[
 Q_i h^2=1+O(\delta).
\]
Hence, for a universal $C>0$ and sufficiently small $\delta$, the fixed collar $[0,L]\times S^2$ lies in the region unchanged by surgery and
\[
 (1-C\delta)g_{\rm cyl}
 \le h^{-2}g(T^\pm)
 \le(1+C\delta)g_{\rm cyl}.
\]
Define $\eta(\delta)\to0$ as
\begin{equation}\label{eq:omega}
 \eta(\delta):=\max\Bigl\{
 \pi\sqrt2\bigl((1+C\delta)^{1/2}-1\bigr),
 L\bigl(1-(1-C\delta)^{1/2}\bigr)\Bigr\}.
\end{equation}

The cross-sections of $g_{\rm cyl}$ have diameter $\pi\sqrt2$, while every path from $\{0\}\times S^2$ to $\{L\}\times S^2$ has length at least $L$. Comparing lengths in the preceding quadratic-form estimate gives \eqref{eq:neck-lengths}. Since $2L>\pi\sqrt2$ and $\eta(\delta)\to0$, decreasing $\delta$ gives \eqref{eq:neck-inequality}.

The finitely many surgery caps lie in disjoint horns \cite[Section~67, Definition~73.1]{KleinerLott2008}. Decreasing $\delta$ keeps the collars in these horns, making the $B_i$ pairwise disjoint. Removing the interiors of the disjoint $3$-balls $D_i$ from connected $Y(T^+)$ leaves $U$ connected. By (ii), $U$ is the part of $X$ in $Y(T^+)$, proving (iii).

Part (iv) follows from \cite[Corollary~5.3(ii)]{DinkelbachLeeb2009}. Indeed, $Y$ is closed, orientable and irreducible, and $|G|>2$ excludes both $S^3$ and $\mathbb{RP}^3$. Thus, if $j_0<j_{\max}$, the last $Y$-component satisfies the extinction alternative in that corollary.

For (v), the scalar evolution equation $\partial_tR=\Delta R+2|\Ric|^2 \ge\Delta R+\frac23R^2$ and the maximum principle give $R\ge b$ on each smooth time interval \cite{Hamilton1982}. At a singular time $T$, smooth convergence gives $R(\cdot,T^-)\ge b(T)$ on $\Omega(T)$, and the bound is unchanged on $X$. If the postsurgery slice is nonempty, a retained component contains a point $q\in\Omega_\rho$, so
\[
 b(T)\le R(q,T^-)\le\rho^{-2}.
\]

The compact truncated standard cap has a positive scalar curvature lower bound. Hence \cite[Lemma~72.20 and its proof, together with Lemma~72.24]{KleinerLott2008}, and $Q_ih^2=1+O(\delta)$, give a universal $c_{\rm cap}>0$ such that $R(\cdot,T^+)\ge c_{\rm cap}h^{-2}$ on each $D_i$. Decrease $\delta$ so that $c_{\rm cap}\delta^{-2}>1$. Since $h<\delta\rho$ by \cite[Lemma~71.1]{KleinerLott2008},
\[
 R(\cdot,T^+)\ge c_{\rm cap}h^{-2}
 >c_{\rm cap}\delta^{-2}\rho^{-2}
 >\rho^{-2}\ge b(T).
\]
Consequently the estimate continues after $T$. Induction over the finitely many singular times in $[0,t]$ proves the two scalar curvature bounds. Since $b(t)\to\infty$ as $t\uparrow T_0$, \cite[Lemma~81.1]{KleinerLott2008} shows that the flow is empty for $t\ge T_0$. This proves (v).

By (i) and (v), the flow has finitely many singular times and vanishes past $T_0$, so $j_0 < j_{\max} < \infty$ in (iv) and the final $Y$-component goes extinct by time $T_0$. Shortly before extinction, the proof of \cite[Theorem~5.1]{DinkelbachLeeb2009} shows this component either has $\sec>0$ or is diffeomorphic to $S^2\times S^1$, $S^3$, $\RP^3$, or $\RP^3\#\RP^3$. Since none of these topologies has finite $|\pi_1|>2$, it must have $\sec>0$ at some time $t_*$, proving (vi).
\end{proof}

The surgery identification and the regions $D_i\subset B_i$, $\Sigma_i$, and $U$ used in Proposition~\ref{prop:ricci-input}(ii)--(iii) are illustrated in Figure~\ref{fig:ricci-surgery}.

\subsection{Systolic estimates}\label{subsec:systolic}

\begin{proposition}\label{prop:smooth-systole}
Let $g(t)$, $t\in[t_-,t_+)$, be a smooth Ricci flow on a closed orientable three-manifold $M$ with $\pi_1(M)\ne1$, and set $s(t)=\sys(g(t))$. Then $s^2$ is locally Lipschitz and
\begin{equation}\label{eq:smooth-systole}
 s(t_2)^2\ge s(t_1)^2-4\pi^2(t_2-t_1),
 \qquad t_-\le t_1\le t_2<t_+.
\end{equation}
\end{proposition}
A dimension-dependent length-decay estimate for a fixed free homotopy class was proved by Ilmanen--Knopf \cite[Lemma~3]{IlmanenKnopf}. Proposition~\ref{prop:smooth-systole} is its sharp three-dimensional version for the homotopy systole; equality is attained by the shrinking round flow on $\RP^3$.
\begin{proof}
On compact time intervals, the metrics are uniformly bilipschitz; hence $s^2$ is locally Lipschitz. At a differentiability time $t_0$, compactness gives a systolic geodesic $\gamma:[0,\ell]\to M$ parametrized by arclength $\tau$, where $\ell=s(t_0)$, and we write $\dot\gamma=d\gamma/d\tau$. For every periodic normal field $V$, small variations remain noncontractible, hence
\begin{equation}\label{eq:second-variation}
 I(V,V)=\int_0^\ell
 \left(|D_\tau V|^2-
 \langle\Rm(V,\dot\gamma)\dot\gamma,V\rangle\right)\,d\tau\ge0.
\end{equation}
Since $M$ is orientable, we can find a parallel oriented orthonormal frame $(E_1,E_2)$ of the normal bundle along $\gamma$. Its holonomy is rotation through an angle $\theta\in[-\pi,\pi]$, so $(E_1(\ell),E_2(\ell))=(E_1(0),E_2(0))\operatorname{Rot}_\theta$. Consider the periodic orthonormal normal fields $(V_1,V_2)(\tau) =(E_1,E_2)(\tau)\operatorname{Rot}_{-\theta\tau/\ell}$, then
\[
 \sum_{i=1}^2\int_0^\ell|D_\tau V_i|^2\,d\tau
 =\frac{2\theta^2}{\ell}\le\frac{2\pi^2}{\ell} \xRightarrow{\text{Sum }\eqref{eq:second-variation} \text{ for } V_1,V_2} \int_0^\ell\Ric(\dot\gamma,\dot\gamma)\,d\tau \le\frac{2\pi^2}{\ell}.
\]
Since $s(t)^2\le L_{g(t)}(\gamma)^2$ with equality at $t_0$,
\begin{align*}
 (s^2)'(t_0)=\left.\frac{d}{dt}\right|_{t=t_0} L_{g(t)}(\gamma)^2 =2\ell\left.\frac{d}{dt}\right|_{t=t_0} L_{g(t)}(\gamma) =-2\ell\int_0^\ell \Ric(\dot\gamma,\dot\gamma)\,d\tau \ge-4\pi^2.
\end{align*}
Integration gives \eqref{eq:smooth-systole}.
\end{proof}
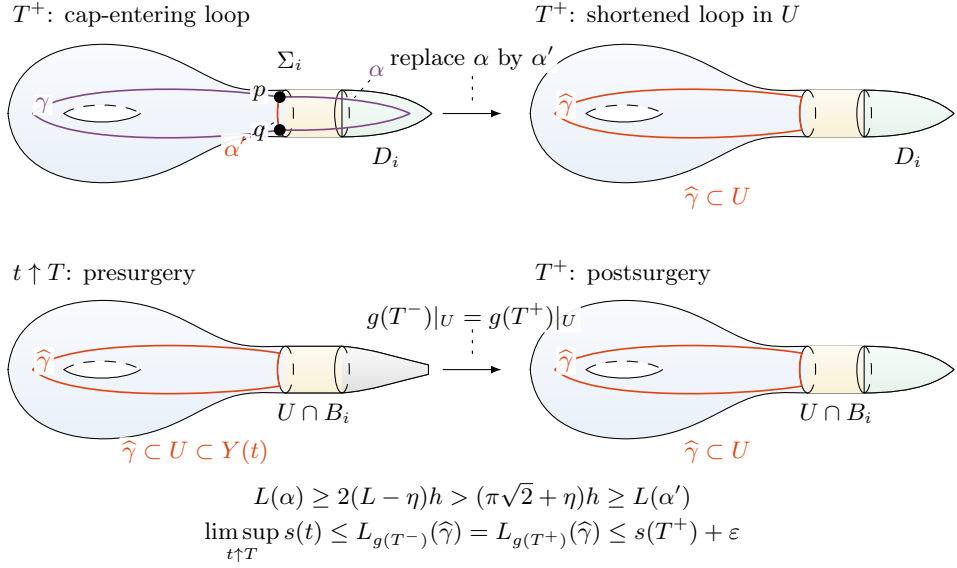
\begin{figure}[ht]
\centering
\begin{tikzpicture}[x=0.88cm,y=0.88cm,font=\scriptsize]

% Upper row: the shortening argument at the postsurgery time T^+.
\begin{scope}[shift={(0,2.65)}]
\systolepostbody
\coordinate (p) at (4.265,.247);
\coordinate (q) at (4.265,-.247);
\draw[loop]
(p) .. controls (3.35,.40) and (1.08,.48) .. (.55,0)
.. controls (1.08,-.48) and (3.35,-.40) .. (q);
\draw[loop]
(p) .. controls (4.92,.27) and (5.82,.20) .. (6.22,0)
.. controls (5.82,-.20) and (4.92,-.27) .. (q);
\draw[shortcut]
(p) arc[start angle=135,end angle=225,
x radius=.12,y radius=.35];

\node[point,minimum size=4pt] at (p) {};
\node[point,minimum size=4pt] at (q) {};
\node[clear label,anchor=east] at (4.10,.30) {$p$};
\node[clear label,anchor=east] at (4.10,-.30) {$q$};
\node[clear label,loop label] at (.70,.13) {$\gamma$};
\node[clear label,loop label] at (5.72,.64) {$\alpha$};
\draw[leader] (5.62,.50) -- (5.39,.24);
\node[clear label,shortcut label] at (3.60,-.49) {$\alpha'$};
\draw[leader] (3.72,-.40) -- (4.24,-.10);
\node[clear label] at (4.42,.79) {$\Sigma_i$};
\node[clear label] at (5.86,-.66) {$D_i$};
\node[clear label,anchor=south west] at (.20,1.25)
{$T^+$: cap-entering loop};
\end{scope}

\begin{scope}[shift={(7.85,2.65)}]
\systolepostbody
\coordinate (ph) at (4.265,.247);
\coordinate (qh) at (4.265,-.247);
\draw[shortcut]
(ph) .. controls (3.35,.40) and (1.08,.48) .. (.55,0)
.. controls (1.08,-.48) and (3.35,-.40) .. (qh)
arc[start angle=225,end angle=135,
x radius=.12,y radius=.35];
\node[clear label,shortcut label] at (.72,.13) {$\widehat\gamma$};
\node[clear label,shortcut label] at (3.00,-1.27)
{$\widehat\gamma\subset U$};
\node[clear label] at (5.86,-.66) {$D_i$};
\node[clear label,anchor=south west] at (.20,1.25)
{$T^+$: shortened loop in $U$};
\end{scope}

\draw[process] (6.72,2.65) -- (7.60,2.65);
\node[clear label,align=center] at (7.16,3.48)
{replace $\alpha$ by $\alpha'$};
\draw[leader] (7.16,2.85) -- (7.16,3.28);

% Lower row: comparison across surgery.
\begin{scope}[shift={(0,-1.20)}]
\systoleprebody
\coordinate (pt) at (4.265,.247);
\coordinate (qt) at (4.265,-.247);
\draw[shortcut]
(pt) .. controls (3.35,.40) and (1.08,.47) .. (.55,0)
.. controls (1.08,-.47) and (3.35,-.40) .. (qt)
arc[start angle=225,end angle=135,
x radius=.12,y radius=.35];
\node[clear label,shortcut label] at (.72,.13) {$\widehat\gamma$};
\node[clear label,shortcut label] at (3.00,-1.24)
{$\widehat\gamma\subset U\subset Y(t)$};
\node[clear label] at (4.78,-.66) {$U\cap B_i$};
\node[clear label,anchor=south west] at (.20,1.22)
{$t\uparrow T$: presurgery};
\end{scope}

\begin{scope}[shift={(7.85,-1.20)}]
\systolepostbody
\coordinate (pu) at (4.265,.247);
\coordinate (qu) at (4.265,-.247);
\draw[shortcut]
(pu) .. controls (3.35,.40) and (1.08,.47) .. (.55,0)
.. controls (1.08,-.47) and (3.35,-.40) .. (qu)
arc[start angle=225,end angle=135,
x radius=.12,y radius=.35];
\node[clear label,shortcut label] at (.72,.13) {$\widehat\gamma$};
\node[clear label,shortcut label] at (3.00,-1.24)
{$\widehat\gamma\subset U$};
\node[clear label] at (4.78,-.66) {$U\cap B_i$};
\node[clear label,anchor=south west] at (.20,1.22)
{$T^+$: postsurgery};
\end{scope}

\draw[process] (6.72,-1.20) -- (7.60,-1.20);
\node[clear label,align=center] at (7.16,-.40)
{$g(T^-)|_U=g(T^+)|_U$};
\draw[leader] (7.16,-.60) -- (7.16,-.98);

\node[clear label] at (7.16,-3.10)
{$L(\alpha)\ge 2(L-\eta)h>(\pi\sqrt2+\eta)h\ge L(\alpha')$};
\node[clear label] at (7.16,-3.78)
{$\displaystyle
\limsup_{t\uparrow T}s(t)
\le L_{g(T^-)}(\widehat\gamma)
=L_{g(T^+)}(\widehat\gamma)
\le s(T^+)+\varepsilon$};

\end{tikzpicture}
\caption{\footnotesize Systolic comparison. \emph{Top:} the arc $\alpha$ entering a cap, with $\alpha\cap\Sigma_i=\{p,q\}$, is replaced at $T^+$ by the shorter arc $\alpha'\subset\Sigma_i$ connecting $p$ and $q$. \emph{Bottom:} under the surgery identification, the region $U\Subset\Omega(T)$ (including the collar $U\cap B_i$) and the loop $\widehat\gamma\subset U$ are unchanged across $T$.}
\label{fig:systolic-surgery-comparison}
\end{figure}
\begin{proposition}\label{prop:surgery-systole}
Let $T$ be a singular time for which $Y(T^+)$ exists. Then
\begin{equation}\label{eq:surgery-systole}
s(T^+)\ge\limsup_{t\uparrow T}s(t).
\end{equation}
\end{proposition}
\begin{proof}
Set $U:=Y(T^+)\setminus\bigcup_i\interior D_i$. By Proposition~\ref{prop:ricci-input}(ii), (iii), $U$ is connected and compactly contained in $\Omega(T)$. Fix $\varepsilon>0$. After a $C^1$-small perturbation, we can choose a smooth noncontractible loop $\gamma\subset Y(T^+)$, transverse to every $\Sigma_i$, with $L_{g(T^+)}(\gamma)\le s(T^+)+\varepsilon/2$.

Since the $B_i$ are three-balls, $\gamma$ is not contained
in any $B_i$. Hence $\gamma^{-1}(\interior B_i)\ne S^1$ for every $i$. For each component $J$ of $\gamma^{-1}(\interior B_i)$ whose image meets $\interior D_i$, let $\alpha=\gamma|_{\overline J}$. Its endpoints lie on $\Sigma_i$, and it contains two disjoint subarcs joining $\partial D_i$ to $\Sigma_i$. Let $\alpha'\subset\Sigma_i$ be a shortest path with the same endpoints. By Proposition~\ref{prop:ricci-input}(iii), \eqref{eq:neck-lengths} and \eqref{eq:neck-inequality},
\begin{align*}
 &L_{g(T^+)}(\alpha)
 \ge 2(L-\eta)h,\\
 &L_{g(T^+)}(\alpha')
 \le(\pi\sqrt2+\eta)h,\\
\Longrightarrow & L_{g(T^+)}(\alpha')-L_{g(T^+)}(\alpha)
 \le\bigl(\pi\sqrt2+\eta-2(L-\eta)\bigr)h<0.
\end{align*}
Since $B_i$ is a three-ball, replacing $\alpha$ by $\alpha'$ is a homotopy relative to the endpoints. Performing all such replacements and smoothing the finitely many corners with length increase less than $\varepsilon/3$ gives a loop $\widehat\gamma\subset U$, such that
\[
 L_{g(T^+)}(\widehat\gamma)\le s(T^+)+\varepsilon.
\]
In particular, $\widehat\gamma$ is noncontractible in $Y(T^+)$.

Capping $U$ by the three-balls $D_i$ gives $Y(T^+)\cong Y$, and the inclusion $U\hookrightarrow Y(T^+)$ induces an isomorphism on fundamental groups. For $t<T$ sufficiently close to $T$, the set $U\Subset\Omega(T)$ lies in a single presurgery component $C$. Cutting $C$ along $\partial U$ and applying van Kampen's theorem gives $\pi_1(C)\cong\pi_1(U)*H$ for some group $H$. By Proposition~\ref{prop:ricci-input}(iv), $C$ is diffeomorphic to either $Y$ or $S^3$. Since $\pi_1(U)\cong\pi_1(Y)\ne1$, the latter is impossible; hence $C=Y(t)$. As $|\pi_1(C)|=|\pi_1(U)|=|G|$, we have $H=1$, and therefore $\pi_1(U) \cong\pi_1(Y(t))$, so $\widehat\gamma$ is noncontractible in $Y(t)$.
Finally, Proposition~\ref{prop:ricci-input}(ii)--(iii) gives smooth convergence on $U$ and $g(T^-)|_U=g(T^+)|_U$. Hence
\begin{align*}
 \limsup_{t\uparrow T}s(t) \le\lim_{t\uparrow T}L_{g(t)}(\widehat\gamma) =L_{g(T^+)}(\widehat\gamma) \le s(T^+)+\varepsilon.
\end{align*}
Letting $\varepsilon\downarrow0$ proves \eqref{eq:surgery-systole}.
\end{proof}
The shortening of $\gamma$ outside the caps and the comparison of $\widehat\gamma$ across the singular time are illustrated in Figure~\ref{fig:systolic-surgery-comparison}.

By Proposition~\ref{prop:ricci-input}(i) and (iv), the $Y$-component uniquely survives the finitely many surgeries preceding any time where $Y(t)$ exists. Since $s(t)^2+4\pi^2t$ is nondecreasing on smooth intervals (Proposition~\ref{prop:smooth-systole}) and across surgeries (Proposition~\ref{prop:surgery-systole}), starting from $s(0)=s_0$ yields
\begin{equation}\label{eq:global-systole}
s(t)^2\ge s_0^2-4\pi^2t
\end{equation}
at every nonsingular time for which $Y(t)$ exists.

\section{Proofs of the main results}\label{sec:proofs}

\begin{proof}[Proof of Theorem~\ref{thm:gap}]
Set $\mu_0=\min_YR_g$ and $s_0=\sys(g)$. Suppose for contradiction that
\begin{equation}\label{eq:gap-assumption}
 \mu_0>0,\qquad \mu_0s_0^2\ge6\pi^2.
\end{equation}
After scaling $g$ to be normalized, Proposition~\ref{prop:ricci-input} applies with $k=g$. For every nonsingular $t<T_0=3/(2\mu_0)$ for which $Y(t)$ exists, Proposition~\ref{prop:ricci-input}(v), \eqref{eq:global-systole}, and \eqref{eq:gap-assumption} give
\begin{equation}\label{eq:product-bound}
 (\min_{Y(t)}R_{g(t)})s(t)^2
 \ge\frac{\mu_0(s_0^2-4\pi^2t)}{1-\frac23\mu_0t}
 \ge6\pi^2.
\end{equation}
Proposition~\ref{prop:ricci-input}(vi) supplies such a time $t_*<T_0$ with $\sec_{g(t_*)}>0$.

Continue $g(t_*)$ by its maximal smooth Ricci flow $\widehat g(t)$, retaining the notation $s(t)=\sys(\widehat g(t))$, and set $\mu(t)=\min_YR_{\widehat g(t)}$. The functions $\mu$ and $s^2$ are locally Lipschitz, and at almost every time the maximum principle and Proposition~\ref{prop:smooth-systole} give $\mu'\ge\frac23\mu^2$, and $(s^2)'\ge-4\pi^2$. Consequently,
\[
 (\mu s^2)'\ge\frac23\mu^2s^2-4\pi^2\mu =\frac23\mu(\mu s^2-6\pi^2).
\]
Since \eqref{eq:product-bound} gives $\mu(t_*)s(t_*)^2\ge6\pi^2$, this yields $\mu(t)s(t)^2\ge6\pi^2$ throughout the smooth flow.

Hamilton's convergence theorem now gives a round limit $\bar g$, after scaling and pullback by diffeomorphisms \cite[Sections~14--17]{Hamilton1982}. Since $(\min R)\sys^2$ is scale invariant and continuous under smooth convergence,
\[
 (\min_YR_{\bar g})\sys(\bar g)^2\ge6\pi^2.
\]
Normalize $\sec_{\bar g}=1$. Then $\sys(\bar g)\ge\pi$. But $(Y,\bar g)=S^3/\Gamma$ with $|\Gamma|>2$, so choose $a\in\Gamma\setminus\{I,-I\}$. Freeness and $a\ne-I$ give an $x\in S^3$ with $0<d_{S^3}(x,ax)<\pi$; the minimizing geodesic from $x$ to $ax$ projects to a noncontractible loop of length less than $\pi$. This contradiction proves the theorem.
\end{proof}

\begin{proof}[Proof of Corollary~\ref{cor:spherical}]
Theorem~\ref{thm:gap} handles $|G|>2$. If $|G|=2$, then $Y\cong\RP^3$, and Proposition~\ref{prop:systolic}(ii) gives the inequality and rigidity. Conversely, the unit round $\RP^3$ has scalar curvature $6$ and systole $\pi$.
\end{proof}

\begin{proof}[Proof of Theorem~\ref{thm:classification}]
For $\kappa=\kappa(F)\le\sqrt{3/2}$, the conformal metric from \cite[ proof of Proposition~2]{ChodoshLi} has $\sec\ge0$. By Hamilton \cite[Theorem~1.2]{Hamilton1986}, $X$ is a spherical space form, a compact quotient of $S^2\times\R$, or a flat manifold.

The flat case is excluded because Bieberbach's theorem \cite[Theorem~3.3.1]{Wolf2011} yields a finite cover $T^3\to X$, contradicting the positive scalar curvature from Proposition~\ref{prop:bridge} via \cite[Theorem~2 and Corollary~2]{SchoenYau1979}.

The orientable quotients of $S^2\times\R$ are $S^2\times S^1$ and $\RP^3\#\RP^3$ \cite[Section~4, pp.~457--458]{Scott1983}. The embedded projective plane in $\RP^3\#\RP^3$ with Propositions~\ref{prop:bridge} and \ref{prop:systolic}(ii) gives
\[
 6\pi^2\le(\inf_XR_{\widetilde g})\sys(\widetilde g)^2\le6\pi^2,
\]
whose rigidity forces $X\cong\RP^3$ and gives a contradiction leaving only $S^2\times S^1$.

If $X$ is spherical with $\pi_1(X)\ne1$, Proposition~\ref{prop:bridge} and Corollary~\ref{cor:spherical} force equality, and rigidity yields $X\cong\RP^3$.

When $\kappa<\sqrt{3/2}$, $\sec>0$ \cite[Proposition~2]{ChodoshLi} makes $X$ spherical \cite[Theorem~1.1]{Hamilton1982}; the strict inequality in Proposition~\ref{prop:bridge} contradicts Corollary~\ref{cor:spherical} if $\pi_1(X)\ne1$, forcing $X\cong S^3$.
\end{proof}

\begin{proof}[Proof of Theorem~\ref{thm:main}]
Let $m\in\{2,3\}$ and $\kappa=\kappa(F)$. By Proposition~\ref{prop:systolic} (applying (ii) to standard $\RP^2\subset\RP^3$), every metric $h$ on $\RP^m$ satisfies
\[
 \bigl(\inf_{\RP^m}R_h\bigr)\sys(h)^2\le m(m-1)\pi^2.
\]
If $\kappa^2<2m/(m+1)$, this contradicts Proposition~\ref{prop:bridge}; thus $\kappa^2\ge2m/(m+1)$. At equality $\kappa^2=2m/(m+1)$, both propositions achieve equality: $\widetilde g$ is round and attains the lower bound in \eqref{eq:conformal}, forcing $|F|\equiv1$. Equalities in \eqref{eq:bridge} and \eqref{eq:defect} then yield
\[
 |F|\equiv1,\qquad g\ \text{is round},\qquad
 H=-mF,\qquad \text{\CYRZH}=\kappa^2,\qquad
 \sys(g)=\frac{2\pi}{\kappa}.
\]
Furthermore, $|A_F(v,v)|^2\le\kappa^2$ and $\text{\CYRZH}=\kappa^2$ imply $|A_F(v,v)|=\kappa$ for all unit $v$. Each prime $g$-geodesic has length $2\pi/\kappa$, and its Euclidean image has constant curvature $\kappa$ and total curvature $2\pi$, forming a plane circle by Fenchel equality. Thus $F$ is planar-geodesic, which Little and Sakamoto classify as the Veronese embedding \eqref{eq:veronese} \cite{Little,Sakamoto}; the converse follows from direct computation.
\end{proof}

\begin{proof}[Proof of Corollary~\ref{cor:spherical-immersion}]
If $\kappa(F)^2<3/2$, the strict inequality in Proposition~\ref{prop:bridge} contradicts Corollary~\ref{cor:spherical}, so $\kappa(F)^2\ge3/2$. At equality, Proposition~\ref{prop:bridge} and Corollary~\ref{cor:spherical} force $Y\cong\RP^3$, and Theorem~\ref{thm:main} yields the Veronese immersion.
\end{proof}

\end{document}